\documentclass[12pt,bezier]{article}

\usepackage{amsmath, amssymb, amsfonts, amsthm,url}
\usepackage{xcolor,colortbl}
\usepackage{graphics}
\usepackage{array}
\usepackage{makecell}
\usepackage{tikz}
\usepackage{tkz-graph}
\usepackage{tkz-berge}
\usepackage{color}
\usepackage[utf8]{inputenc}
\usepackage{graphicx}
\usepackage{hyperref}
\usepackage{url}
\usepackage{amssymb,amsmath,amsthm,latexsym,tikz,url,float,subfig,caption,graphicx,pgfplots,mathrsfs,hyperref}
 \usetikzlibrary{matrix}
\tikzstyle{vertex}=[circle, draw, inner sep=0pt, minimum size=3pt]
\newcommand{\vertex}{\node[vertex]}
\usepackage{color}
\newtheorem{theorem}{Theorem}[section]
\newtheorem{lemma}[theorem]{Lemma}
\newtheorem{pro}[theorem]{Proposition}
\theoremstyle{definition}
\newtheorem{definition}[theorem]{Definition}
\newtheorem{example}[theorem]{Example}
\newtheorem{corollary}[theorem]{Corollary}

\newtheorem{claim}[theorem]{Claim}

\newtheorem{property}[theorem]{Property}

\newcommand{\BC}{\mathrm{BC}}
\newcommand{\Cay}{\mathrm{Cay}}
\newcommand{\dotcup}{\mathbin{\dot\cup}}

\begin{document}
\title{Subgroup Perfect Codes in Bi-Cayley Graphs}
\author{ Masoumeh Koohestani$^{ a}$
\quad Nasrin Soltankhah$^{ b,}$\thanks{Corresponding author} \quad
 Mohsen Ghasemi$^{ c}$ \\
\\[.3cm]
{\normalsize $^{ a,b}$Department of Mathematics, Alzahra University, }\\
{\normalsize Vanak Square 19834, Tehran, Iran}\\
{\normalsize $^{ c}$Department of Mathematics, Urmia University}\\
{\normalsize  P. O. Box 575615-1818, Urmia, Iran}\\
{ $^{ a}$m.koohestani@alzahra.ac.ir,} {  ${ ^b}$soltan@alzahra.ac.ir,}\\
{ $^{ c}$m.ghasemi@urmia.ac.ir, }
}
\maketitle
\begin{abstract} 
In this paper, we study perfect codes in bi-Cayley graphs, focusing on when two subgroups $H$ and $K$ of a group $G$ can form a perfect code of the form $H_0\cup K_1$. We first consider the problem for a general group $G$ and then examine what happens for special classes of groups, particularly abelian groups. Along the way, we obtain several results, including a necessary and sufficient condition for normal subgroups to form a perfect code, as well as an equivalent condition for abelian groups in terms of their Sylow \(2\)-subgroups. Finally, we study the same problem for dihedral groups, providing results for this important class of non-abelian groups.

\end{abstract}
\noindent{\bf AMS Mathematics Subject Classification\,(2020):}   05C25, 05C69, 94B99
\section{Introduction}
Throughout this paper, all graphs and groups are finite, with graphs assumed to be undirected.

Let $G$ be a finite group with identity element $1_G$.
Let $R$, $L$, and $S$ be subsets of $G$ such that $R=R^{-1}$ and $L=L^{-1}$, and $1_G\notin R\cup L$. The {\em bi-Cayley} graph $\BC(G; R, L, S)$ is a graph with vertex set $G_0\cup G_1$, where 
$$G_0=\{g_0~|~ g\in G\}, ~~G_1=\{g_1~|~ g\in G\}$$
and two vertices $g_i$ and $h_j$ where $g, h \in G$ and $i,j \in \{0,1\}$, are adjacent if and only if one of the following conditions holds:
\begin{itemize}
  \item[(1)] $i=j=0$ and $hg^{-1}\in R$;
  \item [(2)] $i=j=1$ and $hg^{-1}\in L$;
  \item [(3)] $i=0,j=1$ and $hg^{-1}\in S$.
\end{itemize}
It should be noted that in the literature, the term {\em semi-Cayley} graph is sometimes used in place of bi-Cayley graph. Bi-Cayley graphs (or semi-Cayley graphs) have been widely studied, see for example \cite{Conder-2020, Feng-2016}. In contrast, perfect codes in bi-Cayley graphs have received relatively little attention to date.

Let $\Gamma$  be a graph with vertex set $V (\Gamma)$ and edge set $E(\Gamma)$.  A \textit{perfect code}
 in $\Gamma$ is a subset $C$ of $V(\Gamma)$ such that no two vertices in $C$ are adjacent and every vertex in $V(\Gamma)\setminus C$ is adjacent to exactly one vertex in $C$. In graph theory, perfect codes are also referred to as \textit{efficient dominating sets} \cite{Dejter-2003,Knor-2012} or  \textit{independent perfect dominating sets} \cite{Lee-2001}.
 The systematic study of perfect codes began with Biggs in 1973 \cite{Biggs-1973}, who considered their existence in distance-transitive graphs. Since then, extensive research has addressed perfect codes. Among other graphs, Cayley  and Cayley sum graphs have received significant attention \cite{Chen-2020, Koohestani-2026, Wang-2024, Zhang-2024, Zhou-2016}. In contrast, despite their structural resemblance to Cayley graphs, the  perfect codes of bi-Cayley graphs remain comparatively unexplored and only limited results are known for perfect codes in these graphs up to now \cite{Koohestani-2026-3, Wang-2021, Wang-2026}.

Let $G$ be a finite group and let $H,K\subseteq G$. The set $H_0\cup K_1$ is called a {\em perfect code} of $G$ if it occurs as a perfect code in some bi-Cayley graph over $G$. In the special case where both $H$ and $K$ are subgroups of $G$, we call $H_0\cup K_1$ a {\em subgroup perfect code} of $G$. When $H=K$, we simply refer to $H$ as a {\em subgroup perfect code}. To illustrate this terminology, see Figure~\ref{FigD6}, where $\{1,b\}$ is a subgroup perfect code of $D_6$.

\begin{figure}[t]
	\centering
	\begin{tikzpicture}[scale=.8]
		\vertex[fill] at (2,0) [label=above:\scriptsize{$(a^2b)_0$}] {};
		\vertex[fill]  at (4,0) [label=above:\scriptsize{$(a)_0$}] {};
		\vertex[fill]  at (6,0) [label=above:\scriptsize{$(ab)_0$}] {};
		\vertex[fill] at (8,0) [label=above:\scriptsize{$(a^2)_0$}] {};
		\vertex[fill] at (2,-2) [label=below:\scriptsize{$(a^2b)_1$}] {};
		\vertex[fill] at (4,-2) [label=below:\scriptsize{$(a)_1$}] {};
		\vertex[fill] at (6,-2) [label=below:\scriptsize{$(ab)_1$}] {};
		\vertex[fill] at (8,-2) [label=below:\scriptsize{$(a^2)_1$}] {};
\tikzstyle{vertex}=[circle, draw, blue, inner sep=0pt, minimum size=4pt]
 \vertex[fill] at (0,0) [label=above:\scriptsize{$(1)_0$}] {};
 \vertex[fill]  at (10,0) [label=above:\scriptsize{$(b)_0$}] {};
 \vertex[fill]  at (0,-2) [label=below:\scriptsize{$(1)_1$}] {};
 \vertex[fill]  at (10,-2) [label=below:\scriptsize{$(b)_1$}] {};
\draw[line width=1pt]  (0,0) --  (2,0);
\draw[line width=1pt]  (4,0) --  (6,0);
\draw[line width=1pt]  (8,0) --  (10,0);
\draw[line width=1pt]  (0,-2) --  (2,-2);
\draw[line width=1pt]  (4,-2) --  (6,-2);
\draw[line width=1pt]  (8,-2) --  (10,-2);
\draw[line width=1pt]  (0,0) --  (6,-2);
\draw[line width=1pt]  (2,0) --  (8,-2);
\draw[line width=1pt]  (4,0) --  (10,-2);
\draw[line width=1pt]  (6,0)--(0,-2);
\draw[line width=1pt]  (8,0)--(2,-2);
\draw[line width=1pt]  (10,0)--(4,-2);
\end{tikzpicture}
\caption{$\BC(D_6,\{a^2b\}, \{a^2b\}, \{ab\})$}\label{FigD6}
\end{figure}
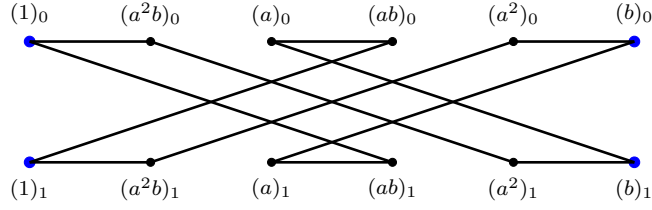
The paper is organized as follows. In Section~\ref{Prel}, we introduce the notation and recall the basic definitions and results on bi-Cayley graphs that we will need later. In Section~\ref{Subgroup}, we study when two normal subgroups of a group $G$ can form a perfect code in a bi-Cayley graph over $G$. We also give a necessary and sufficient condition for their existence in terms of left transversals. Section~\ref{Abeli} focuses on subgroup perfect codes in abelian groups, particularly on the connection between their existence in an abelian group and in its Sylow $2$-subgroup. In Section~\ref{Dihedral}, we construct several subgroup perfect codes in dihedral groups. Finally, Section~\ref{Conclude} concludes the paper by summarizing our main results and discussing some possible directions for future research.

 \section{Preliminaries}\label{Prel}
Let $\Gamma$  be a graph with vertex set $V (\Gamma)$ and edge set $E(\Gamma)$. We write $u\sim v$, where $u,v\in V(\Gamma)$, to indicate that $u$ and $v$ are adjacent. 
 We use $\Gamma (u)$ 
 for the neighbourhood of a vertex $u\in V(\Gamma)$, that is, the set of vertices adjacent to $u$ in $\Gamma$. We  also use $\dotcup_{i=0}^nS_i$ to denote the union of the pair-wise disjoint sets $S_1, S_2, ..., S_n$.
 
Suppose that $G$ is a finite group. For any subsets $U, V\subseteq G$, define
$$UV=\{uv~|~u\in U,~ v\in V\}$$
whenever both $U$ and $V$ are nonempty, and $UV=\emptyset$ when either $U$ or $V$ is empty. In particular, if $U=\{u\}$ or $V=\{v\}$, then we write $uV$ or $Uv$, respectively. Evidently, we have $UV\subseteq G$ and 
$UV=\cup _{u\in U}uV$.
This union is not necessarily disjoint. We therefore define
$$U*V:=\dotcup_{_{u\in U}}uV.$$
So for a subgroup $H$ of $G$ we have $UH=U*H$ if and only if $U$ is a left transversal of $H$ in $UH$. We note that the left transversal of $H$ in $UH$ is defined in the same way as the left transversal of $H$ in $G$. Throughout this paper, we use the notation $\dotcup$ whenever it is necessary to indicate that the union is disjoint.  

Suppose that $\Gamma=\BC(G;R, L, S)$ is a bi-Cayley graph.  The neighborhood of vertices $x_0, x_1\in V(\Gamma)$ are obtained as follows: 
   \begin{align}\label{Neighbor}
     \Gamma(x_0)&=\{(rx)_0, (sx)_1~|~ r\in R, s\in S\}=(Rx)_0\cup(Sx)_1, \nonumber\\
      \Gamma(x_1)&= \{(lx)_1, (s^{-1}x)_0~|~ l\in L, s\in S\}=(Lx)_1\cup(S^{-1}x)_0.
   \end{align}
The next proposition follows directly from the definition of perfect codes and so, we omit the proof.
\begin{pro}\label{Partition}
Let $\Gamma$ be a graph and $C$ be a subset of the vertices of $\Gamma$. If $C$ is a  perfect code of $\Gamma$, then the $|C|+1$ sets $\{\Gamma(c)~|~ c\in C\}$ and $C$ constitute a partition  of $V(\Gamma)$.
\end{pro} 
Applying \eqref{Neighbor} and Proposition \ref{Partition} to a perfect code $C=H_0 \cup K_1$ in the bi-Cayley graph $\Gamma=\BC(G;R, L, S)$ yields 
\begin{align}\label{Vertex-Perfect}
  V(\Gamma)&=(R*H)_0 \dotcup (S*H)_1 \dotcup (L*K)_1 \dotcup (S^{-1}*K)_0 \dotcup H_0 \dotcup K_1.
\end{align}
The relationship between perfect codes in bi-Cayley graphs and left transversals of a group follows from \eqref{Vertex-Perfect}, analogously to that for Cayley graphs \cite{Huang-2018}. A detailed proof of this relationship is provided in \cite{Weiyue-2026}. 
\begin{lemma}\cite{Weiyue-2026}\label{LeftTrans-Lem}
Let $\Gamma=\BC(G;R, L, S)$ be a bi-Cayley graph and $H$ and $K$ be  subgroups of $G$. Then, $H_0\cup K_1$ is a subgroup perfect code of $G$ if and only if $R\cup \{1_G\}$ is a left transversal of $H$ in $G\setminus S^{-1}*K$ and $L\cup \{1_G\}$ is a left transversal of $K$ in $G\setminus S*H$.
  \end{lemma}
  The following result is an immediate consequence of Lemma \ref{LeftTransversal}.
  \begin{corollary}\label{LeftTransversal}
    Let $\Gamma=\BC(G;R, L, S)$ be a bi-Cayley graph and $H$ be a subgroup of $G$. Then, $H$ is a subgroup perfect code of $G$ if and only if both $R\cup S^{-1} \cup \{1_G\}$ and $L\cup S \cup \{1_G\}$ are left transversals of $H$ in $G$.
      \end{corollary}
 
\section{Subgroup Perfect Codes}\label{Subgroup}
In order to investigate subgroup perfect codes of finite groups, we also collect several facts and observations that facilitate the study the structure of bi-Cayley graphs.  Throughout this section, $G$ is a finite group, $|G|$ is the cardinality of $G$ and  $o(g)$ denotes the order of $g$, where $g\in G$.
\begin{claim}\label{SInterKH}
Let $H$ and $K$ be subgroups of $G$. If $H_0\cup K_1$ is a perfect code in a bi-Cayley graph $\BC(G;R,L,S)$ over $G$, then
$S\cap KH=\emptyset$.
 \end{claim}
\begin{proof}
  Suppose, to the contrary, that there exists an element $s\in S\cap KH$, then $s=kh$ for some $k\in K $ and $h\in H$. Hence, $sh^{-1}=k$ which implies that $(h^{-1})_0\sim k_1$, contradicting with the independency of $H_0\cup K_1$. 
\end{proof}
An immediate consequence of Claim \ref{SInterKH} is that if $G=KH$, where $H$ and $K$ are subgroups of $G$, then neither of $H_0\cup K_1$ nor $H_1\cup K_0$ can be perfect code of any connected bi-Cayley graph, since $S$ must be empty. 
\begin{definition}\label{S-S-Inverse}
Let $H$ and $K$ be subgroups of $G$. We define the subset $\mathcal{S}$ as follows:
$$\mathcal{S}=\{S\subseteq G~|~ S\cap KH=\emptyset,~  S^{-1}K=S^{-1}*K,~ SH=S*H\}.$$ 
Note that $\emptyset \in \mathcal{S}$, and so $\mathcal{S}$ is  nonempty.
\end{definition}
In view of Definition \ref{S-S-Inverse}, let $S\in\mathcal{S}$ be fixed for two given subgroups $H$ and $K$ of $G$. Suppose that there exist inverse-closed subsets $R,L\subseteq G$ such that
\begin{equation}\label{Construction}
R*H \cup S^{-1}*K\cup H=L*K \cup S*H\cup K=G.
\end{equation}
Then, by Lemma \ref{LeftTrans-Lem}, $H_0\cup K_1$ is a subgroup perfect code of $G$
 The following example illustrates this construction.
\begin{example}
Consider the dihedral group
$D_{12}=\langle a,b\mid a^6=b^2=1, ab=ba^{-1}\rangle$
and let $H=\langle b\rangle$ and $K=\langle a^3b\rangle$ be subgroups of $D_{12}$. Set $S=\{a,a^2,a^4,a^5\}$.
It is readily verified that $S\in\mathcal{S}$. Set $R=\{a^3b\}$ and $L=\{a^3\}$. Since \eqref{Construction} holds for $\BC(D_{12};R,L,S)$, it follows that $H_0\cup K_1$ is a subgroup perfect code in $D_{12}$.
\end{example} 
When subgroups $H$ and $K$ are normal in $G$, we prove that the existence of inverse-closed subsets $R$ and $L$ for a given subset $S\in \mathcal{S}$ satisfying \eqref{Construction} is equivalent to the following property.
\begin{property}\label{Property}
Let $H$ and $K$ be  normal subgroups of $G$. Suppose that $S\in \mathcal{S}$, and set $C=H_0 \cup K_1$. For any $g_i\in G_0\cup G_1\setminus (S^{-1}*K)_0\cup (S*H)_1$, if $(g^2)_i\in C$, then there exists some $x_i\in C$ such that $(gx)^2=1_G$. In another words,
\begin{equation*}
  \forall g_i\in (G_0\cup G_1)\setminus ((S^{-1}*K)_0 \cup(S*H)_1),  ~~ (g^2)_i\in C~\Rightarrow~ \exists x_i\in C: ~~ (gx)^2=1_G .
\end{equation*}
\end{property}
\begin{theorem}\label{Perfect-Property} 
Let $H$ and $K$ be a normal subgroups of $G$.
Then, $H_0\cup K_1$ is a subgroup perfect code of $G$ if and only if the property \eqref{Property} holds.
\end{theorem}
\begin{proof}
Let $S\in \mathcal{S}$ and $C=H_0\cup K_1$ be a subgroup perfect code in a bi-Cayley graph of $G$ such as $\Gamma=\BC(G;R,L,S)$. Suppose that $g_i\in G_0\cup G_1\setminus (S^{-1}*K)_0 \cup(S*H)_1$ with $(g^2)_i\in C$.  We continue the proof in the following two cases.
\begin{itemize}
  \item [(i)] We have $(g^2)_i\in H_0$.  From Lemma \ref{LeftTransversal}, $R\cup \{1_G\}$ is a left transversal of $H$ in $G\setminus S^{-1}*K$. So, there exists a unique $r\in R\cup \{1_G\}$ such that $g\in rH$. Therefore,  $gH=rH$ and since $H$ is normal in $G$, we have
       $$g^{-1}H=Hg^{-1}=Hr^{-1}=r^{-1}H.$$
       On the other hand, $g^2\in H$. Hence, we have $gH=g^{-1}H$ and $rH=r^{-1}H$. Since $r, r^{-1}\in R=R^{-1}$ and $R$ is a left transversal of $H$ in $G\setminus S^{-1}*H$, we conclude that $r=r^{-1}$ and $r^2=1_G$. This together with $g\in rH$ imply that there is some $h\in H$ for which we have $(gh)^2=1_G$.
  \item [(ii)] If $(g^2)_i\in K_1$. Then, $g^2\in K$ and $g\in G\setminus S*H$.  From Lemma \ref{LeftTransversal}, $L\cup \{1_G\}$ is a left transversal of $K$ in $G\setminus S*H$. So, there exists a unique $l\in L\cup \{1_G\}$ such that $g\in lK$. Therefore,  $gK=lK$ and since $K$ is normal in $G$, we have
       $$g^{-1}K=Kg^{-1}=Kl^{-1}=l^{-1}K.$$
        As $g^2\in K$, we have $gK=g^{-1}K$ and $lK=l^{-1}K$. Since $l, l^{-1}\in L=L^{-1}$ and $L\cup \{1_G\}$ is a left transversal of $K$ in $G\setminus S*H$, we conclude that $l=l^{-1}$ and $l^2=1_G$. This together with $g\in lK$ imply that there is some $k\in K$ for which we have $(gk)^2=1_G$.
\end{itemize} 
To prove the sufficiency, suppose that $H$ and $K$ are normal subgroups of $G$  and $S\in \mathcal{S}$ and Property \eqref{Property} holds. We now construct inverse-closed subsets $R,L \subseteq G$ such that \eqref{Construction} holds. 
\begin{itemize}
  \item 
  We consider two following sets:
  \begin{align*}
    I_1= & \{x_1H,\dots, x_rH~|~\forall~1\leqslant i\leqslant r:~  x_iH\notin S^{-1}*K, ~ x_iH=(x_iH)^{-1}\}, \\
    I_2= & \{y_1H,\dots, y_sH~|~\forall~1\leqslant i\leqslant s:~y_iH\notin S^{-1}*K, ~ y_iH\neq(y_iH)^{-1}\}.
  \end{align*}
    For each $x_iH\in I_1$, where $1\leq i\leq r$, pick an element $x_ih \in x_iH$. 
     We have $(x_ih)_0\in G_0\setminus (S^{-1}*K)_0$ and $(x_i^2)_0\in H_0$.
     Since $x_i^2\in H$, there exists an element $h^{\prime}\in H$ such that $x_i=x_i^{-1}h^{\prime}$ and so, $(x_ih)^2=x_ihx_i^{-1}h^{\prime}h\in H$ because $H$ is normal in $G$.
      So according to property \ref{Property}, there exists $h_i\in H$ such that $(x_ihh_i)^2=1_G$. Set $z_i:=x_ihh_i$. Note that $z_i H=x_iH\notin S^{-1}*K$. Finally, we construct the inverse-closed set $R$ as follows:
    $$R=\{z_1, \dots, z_r,y_1, y_1^{-1},\dots, y_s, y_s^{-1}\}$$ 
    which is a left transversal of $H$ in $G\setminus S^{-1}*K$.
 \item 
 We define two sets as follows:
  \begin{align*}
    I_1= & \{x_1K,\dots, x_rK~|~\forall~1\leqslant i\leqslant r:~  x_iK\notin  S*H, ~ x_iK=(x_iK)^{-1}\}, \\
    I_2= & \{y_1K,\dots, y_sK~|~\forall~1\leqslant i\leqslant s:~y_iK\notin S*H, ~ y_iK\neq(y_iK)^{-1}\}.
  \end{align*}
    For each $x_iK\in I_1$, pick an element $x_ik \in x_iK$. 
     We have $(x_ik)_0\in G_0\setminus (S*H)_0$ and $(x^2_i)_0\in K_0$.
     Since $x_i^2\in K$, there exists an element $k^{\prime}\in K$ such that $x_i=x_i^{-1}k^{\prime}$ and so, $(x_ik)^2=x_ikx_i^{-1}k^{\prime}k\in K$ because $K$ is normal in $G$.
      So according to property \ref{Property}, there exists $k_i\in K$ such that $(x_ikk_i)^2=1_G$. Set $z_i:=x_ikk_i$. Note that $z_i K=x_iK\notin S*H$.  Finally, we construct the inverse-closed set $L$ as follows:
    $$L=\{z_1, \dots, z_r,y_1, y_1^{-1},\dots, y_s, y_s^{-1}\}.$$ 
 \end{itemize} 
Hence, $H_0\cup K_1$ is a perfect code in $\BC(G, R, L, S)$ and the proof is completed.
\end{proof}
\begin{corollary}\label{OddOrder}
  Let $H$ and $K$ be normal subgroups of $G$. If either $|H|$ or $|G/H|$ and $|K|$ or $|G/K|$ is odd, then $H_0\cup K_1$ is a subgroup perfect code of $G$.
\end{corollary}
\begin{proof}
We let $S\in \mathcal{S}$ and $C=H_0\cup K_1$. We prove that Property \ref{Property} holds. Suppose that  $g_i\in G_0\cup G_1\setminus (S^{-1}*K)_0 \cup(S*H)_1$ with $(g^2)_i\in C$. Therefore, $g^2\in H$ or $g^2\in K$. Suppose that $g^2\in H$. We continue the proof in the following two cases.
\begin{itemize}
  \item Assume that $|H|$ is odd. Then, $(g^2)^{2k+1}=1_G$ for some integer $k$. Define $h:=g^{2k}$. since  $g^{2k}=(g^2)^k$, it follows that 
      $h\in H$. Moreover,
      $$(gh)^2=g^{4k+2}=(g^2)^{2k+1}=1_G.$$
  \item Let $|G/H|$ be odd. Since $g^2\in H$, it follows that $(gH)^2=H$. Therefore, $gH=H$ because $o(gH)\leqslant 2$ and $o(gH)$ divides $|G/H|$ which is an odd number. Consequently, $g\in H$. Set $h:=g^{-1}\in H$. Evidently, we have $(gh)^2=1_G$.
\end{itemize}
 In the case that $g^2\in K$,  by a similar argument to when $g^2\in H$, we conclude that there exists some $k\in K$ such that $(gk)^2=1_G$.  Therefore, Property \ref{Property} holds and $C$ is a perfect code of $G$.  
 \end{proof}
 Note that Corollary \ref{OddOrder} does not guarantee the connectedness of the bi-Cayley graph over $G$ in which $H_0\cup K_1$  is a perfect code. For instance, in the following  proposition, we prove that $\{e_0\}\cup H_1$, where $H$ is  a nontrivial subgroup of  an abelian group $G$, cannot be perfect code in any conncted bi-cayley graph, even though the conditions in the corollary are satisfied.
 \begin{pro}\label{AbeliSEmpty}
 Let $G$ be an abelian group and $H$ be a non-trivial subgroup of $G$.  If $\{e_0\}\cup H_1$ is a perfect code in a bi-Cayley graph $\BC(G;R, L, S)$, then $S=\emptyset$. 
 \end{pro}
 \begin{proof}
   Let $C=\{e_0\}\cup H_1$ be a perfect code in a bi-Cayley graph $\BC(G;R, L, S)$. Suppose, to the contrary, that $S\neq \emptyset$.  We firstly prove that $S=S^{-1}$. By \eqref{Construction},  we obtain
   $$G=L*H\dotcup S\dotcup H.$$
   Since $G$ is abeli and $L$ is inverse-closed, $G^{-1}=G=L*H\dotcup S^{-1}\dotcup H$ and hence, $S=S^{-1}$. Since $S\neq \emptyset$, we let $s, s^{-1}\in S$. We obtain 
   \begin{equation}\label{ss-1}
    s_0\sim e_1~~\text{and}~~s_1\sim e_0.
   \end{equation}
   Take $e\neq h\in H $. Note that $hs\notin H$ because otherwise we would have $S\cap H\neq \emptyset$ which contradicts Claim \ref{SInterKH}. Therefore, $(hs)_1$ is outside $C$ and must have a neighbor in $C$. If $(hs)_1\sim e_0$, then we have $s^{\prime}hs=e$ and $s^{\prime}s=h^{-1}$, for some $s^{\prime}\in S$. Consequently, we obtain $s_0\sim (h^{-1})_1$, a contradiction with \eqref{ss-1}. If $(hs)_1\sim h^{\prime}_1$ for $h^{\prime}\in H$, then $lhs=h^{\prime}$ and $ls=h^{\prime}h^{-1}$. So, $s_1\sim (h^{\prime}h^{-1})_1$, contradicting \eqref{ss-1}. Therefore, $(hs)_1$ has no neighbor in $C$.
   \end{proof}
 \begin{corollary}
 Let $G$ be of odd cardinality. Then, for normal subgroups $H$ and $K$ of $G$, $H_0\cup K_1$ is a subgroup perfect code.
 \end{corollary}
 \begin{proof}
   Since $|H|$ and $|K|$ are divisors of $|G|$, $|H|$ and $|K|$ must be odd. So, it follows from Corollary \ref{OddOrder} that $H_0 \cup K_1$ is a subgroup perfect code of $G$. 
 \end{proof}
 \subsection{Subgroup Perfect Codes of Abelian Groups}\label{Abeli}
 Here, we investigate subgroup perfect codes in abelian groups. We begin with two claims that will be used to establish the main results of this section. We omit the proof of the following claim, as it is straightforward.
  \begin{claim}\label{P-Sylow-Intersection}
 Let $G$ be an abelian group and $P$ be a Sylow 2-subgroup of $G$. Then, every subgroup $H$ of $G$ can be written as 
$$H=(H\cap P)\times (H\cap Q),$$
where $Q$ is the Hall $2^{\prime}$-subgroup of $G$. 
\end{claim}
 \begin{claim}
    Suppose that $H$ is a subgroup of $G$ and $S$ is a subset of $G$. If $P$ is a Sylow $2$-subgroup of $G$, then we have 
   $$P \cap (S*H) \subseteq (S\cap P)*(H\cap P).$$
 \end{claim}
 \begin{proof}
 We know that $G=P\times Q$, where $Q$ is the Hall $2^{\prime}$-subgroup of $G$. Let $x\in P\cap (S*H)$. Then, $x=sh$, for some $s\in S$ and $h\in H$. Moreover, $s= s_ps_q$ and $h=h_ph_q$ with $s_p,h_p\in P$ and $s_q,h_q\in Q$.
 Since $x\in P$, we must have $s_qh_q=1$. On the other hand, $s_p\in S\cap P$ and $h_p\in H\cap P$. Therefore, $x=s_ph_p\in (S\cap P)*(H\cap P)$. 
 \end{proof}
\begin{lemma}\label{SylowLem}
  Let $H$ and $K$ be subgroups of $G$ and  $P$ be a Sylow 2-subgroup of $G$. If $H_0 \cup K_1$  is a subgroup perfect code of $G$, then $(H\cap P)_0\cup (K\cap P)_1 $ is a subgroup perfect code of $P$.
\end{lemma}
\begin{proof}
 We know that $G=P\times Q$, where $Q$ is the Hall $2^{\prime}$-subgroup of $G$. 
  Suppose that $H_0\cup K_1$ is a subgroup perfect code of $G$. Set
   $S\in \mathcal{S}$,
  $S_p=S\cap P$, $H_p=H\cap P$, $K_p=K\cap P$, and $C_p=(H_p)_0 \cup (K_p)_1$. We want to prove that the following statement holds:
  \begin{equation}\label{Sylow}
    \forall g_i\in (P_0\cup P_1)\setminus ((S_p^{-1}*K_p)_0 \cup(S_p*K_p)_1),  ~ (g^2)_i\in C_p~\Rightarrow~\exists x_i\in C_p:~ ~(gx)^2=1_G.
  \end{equation} 
   Assume that $g_i\in P_0\cup P_1\setminus (S_p^{-1}*K_p)_0 \cup(S_p*H_p)_1$ with $(g^2)_i\in C_p$. Note that $g_i\notin (S^{-1}*K)_0 \cup(S*H)_1$ because otherwise $g\in P\cap  S^{-1}*K$ or $g\in P\cap S*H$. This together with Claim \ref{P-Sylow-Intersection} imply that $g\in S_p^{-1}*K_p$ or $g\in S_p*H_p$ which contradicts the assumption in \eqref{Sylow}. Therefore, by Property \ref{Property}, there exists some $x_i\in C$ such that $(gx)^2=1_G$. We note that  $x^{-2}=g^{-2}\in H_p$ when $i=0$, and $x^{-2}\in K_p$ if $i=1$, and so, $x\in H_p$ when $i=0$, and $x\in K_p$ if $i=1$, which implies that $x_i\in C_p$.
   \end{proof}
 Note that the condition in Lemma \ref{SylowLem} is necessary but not sufficient in general, particularly, in the context of   connected bi-Cayley graphs. The following example illustrates this fact.
  \begin{example}
    Let $G=\langle a \rangle \times \langle b \rangle \times  \langle c \rangle$, where $o(a)=4$, $o(b)=3$, and $o(c)=5$. The Sylow 2-subgroup of $G$ is $P=\langle a \rangle$. Let $H=\langle b \rangle$, $K=\langle c \rangle$.  Suppose that 
     $$\Gamma =\BC(P;\{a, a^3\},\{a, a^3\},\{a^2\}),$$
      as shown in Figure \ref{SylowFig}. Then  
      $(H\cap P)_0\cup (K\cap P)_1=\{e_0, e_1\}$
       is a perfect code of $\Gamma$. 
       However, $C=H_0\cup K_1$ cannot be a perfect code in any connected bi-Cayley graph over $G$. Indeed, suppose that there exists a bi-Cayley graph $\BC(G; R, L, S)$ that admits $C$ as a perfect code. By 
     \eqref{Construction}, we must have 
     $$60=3|R|+3+5|S|~~  \text{and}~~ 60=5|L|+5+3|S|.$$ The first equation implies that $|S|\in \{0, 3, 6, 9\}$, whereas the second implies that $|S|\in \{0,5,10,15\}$. Thus, the only possible common value is $|S|=0$. Since $S=\emptyset$ implies that the bi-Cayley graph is disconnected, no connected bi-Cayley graph over $G$ can admit $C$ as a perfect code.
  \end{example}
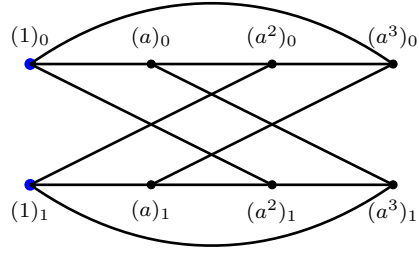
\begin{figure}[t]
	\centering
	\begin{tikzpicture}[scale=.8]
		\vertex[fill]  at (2,0) [label=above:\scriptsize{$(a)_0$}] {};
		\vertex[fill]  at (4,0) [label=above:\scriptsize{$(a^2)_0$}] {};
		\vertex[fill] at (6,0) [label=above:\scriptsize{$(a^3)_0$}] {};
		\vertex[fill]  at (2,-2) [label=below:\scriptsize{$(a)_1$}] {};
		\vertex[fill]  at (4,-2) [label=below:\scriptsize{$(a^2)_1$}] {};
		\vertex[fill] at (6,-2) [label=below:\scriptsize{$(a^3)_1$}] {};
\tikzstyle{vertex}=[circle, draw, blue, inner sep=0pt, minimum size=4pt]
 \vertex[fill] at (0,0) [label=above:\scriptsize{$(1)_0$}] {};
 \vertex[fill] at (0,-2) [label=below:\scriptsize{$(1)_1$}] {};
\draw[line width=1pt]  (0,0) --  (2,0);
\draw[line width=1pt]  (2,0) --  (4,0);
\draw[line width=1pt]  (4,0) --  (6,0);
\draw[line width=1pt]  (0,-2) --  (2,-2);
\draw[line width=1pt]  (2,-2) --  (4,-2);
\draw[line width=1pt]  (4,-2) --  (6,-2);
\draw[line width=1pt, bend left=35]  (0,0) to   (6,0);
\draw[line width=1pt, bend left=-35]  (0,-2) to   (6,-2);
\draw[line width=1pt]  (0,0) --  (4,-2);
\draw[line width=1pt]  (2,0) --  (6,-2);
\draw[line width=1pt]  (4,0) --  (0,-2);
\draw[line width=1pt]  (6,0) --  (2,-2);
\end{tikzpicture}
\caption{$\BC(\langle a\rangle,\{a,a^3\}, \{a,a^3\}, \{a^2\})$}\label{SylowFig}
\end{figure}
\begin{theorem}\label{EquSyl}
  Let $H$ and $K$ be subgroups of $G$ such that $ H\leq K$ and  $P$ be a Sylow 2-subgroup of $G$. Then,  $H_0 \cup K_1$  is a subgroup perfect code of $G$ if and only if $(H\cap P)_0\cup (K\cap P)_1 $ is a subgroup perfect code of $P$.
\end{theorem}
\begin{proof}  
The necessity follows from Lemma \ref{SylowLem}.

    For the sufficiency, suppose that $G=P\times Q$, where $Q$ is the $2^{\prime}$-Hall subgroup of $G$. Set $H_p:=H\cap P$, $H_q:=H\cap Q$, $K_p:=K\cap P$, and $K_q:=K\cap Q$. Let \eqref{Sylow} holds. We find a bi-Cayley graph in which $H_0 \cup K_1$ is a perfect code.
    Let $S\in \mathcal{S}$ and $S$ be inverse-closed.
      Let $g_i\in G_0\cup G_1\setminus (S^{-1}*K)_0 \cup(S*H)_1$ with $(g^2)_i\in H_0\cup K_1$.
  We write $g=g_pg_q$, where $g_p\in P$ and $g_q\in Q$.
  With Claim \ref{P-Sylow-Intersection}, since  $g^2=g_p^2g_q^2\in H\cup K$, it follows that $g_p^2\in H_p\cup K_p$ and $g_q^2\in H_q\cup K_q$. On the other hand, $g_q$ has odd order because $o(g_q)$ divides $|H_q\cup K_q|$, which is odd.  This implies that
  \begin{equation}\label{gq}
    g_q=g_q^2 g_q^{2k}\in H_q\cup K_q.
  \end{equation} 
  We now note that $(g_p)_i\notin(S_p^{-1}*K_p)_0 \cup (S_p*H_p)_1$ because otherwise if  $g_p\in  S_p^{-1}*K_p$, then $g_p=s^{-1}_pk_p$, where $s_p\in S_p$ and $k_p\in K_p$. Using \eqref{gq}, $g=g_pg_q= s^{-1}_pk_p g_q\in S^{-1}*K$ which is a contradiction. Similarly, if we consider $g_p\in  S_p*H_p$, then $g_p=s_ph_p$, where $s_p\in S_p$ and $h_p\in h_p$. This leads to $g=g_pg_q= s_ph_p g_q\in S*K=S^{-1}*K$ by \eqref{gq}, a contradiction.
          We now conclude from \eqref{Sylow} that there exists $(x_p)_i\in (H_p)_0\cup (K_p)_1$  satisfying $(g_px_p)^2=1_G$.
  We now let $x:= x_pg_q^{-1}$ which is an element of $H\cup K$. We have $(gx)^2=(g_pg_qx_pg_q^{-1})^2=1_G$ and so, $H_0\cup K_1$ is a subgroup perfect code of $G$.
\end{proof}
It should be noted that, under the condition in Theorem \ref{EquSyl}, admitting a perfect code in a bi-Cayley graph over a group is equivalent to admitting a perfect code in a bi-Cayley graph over its Sylow \(2\)-subgroup. However, this does not imply that a perfect code necessarily exists under this condition in general; rather, the existence of such a perfect code may require additional conditions, as in the special case described below.
\begin{pro}
  Let $G$ be an abelian group and $H$ and $K$ be subgroups of $G$ such that $H\leq K$. If $H_0 \cup K_1$ is a perfect code of a bi-Cayley graph $\BC(G;R, L, S)$, where $S$ is nontrivial and $S=S^{-1}$, then $H=K$. 
\end{pro}
\begin{proof}
  Let $C=H_0 \cup K_1$ is a perfect code of a bi-Cayley graph $\BC(G;R, L, S)$ with  $S\neq \emptyset$ and $S=S^{-1}$. We let $s, s^{-1}\in S$. We obtain 
   \begin{equation}\label{s0s1}
    s_0\sim e_1\in K_1~~\text{and}~~s_1\sim e_0\in H_0.
   \end{equation}
   Take $e\neq k\in K$. Note that $(ks)_1\notin C$ because otherwise we would have $S\cap K\neq \emptyset$ which contradicts Claim \ref{SInterKH}. Therefore, $(ks)_1$ must have a neighbor in $C$. If $(ks)_1\sim k^{\prime}_1$ where $k^{\prime}\in K$, then $lks=k^{\prime}$. This leads to $s_1\sim  (k^{\prime}k^{-1})_1\in C$, contradicting \eqref{s0s1}. We now suppose that $(ks)_1\sim h_0$ where $h\in H$, then $s^{\prime}ks=h$ for some $s^{\prime}\in S$. So, we obtain $s_0\sim (hk^{-1})_1$ and because of \eqref{s0s1}, we must have $k=h\in H$. Therefore, $K\leq H$.
\end{proof}

In \cite[Lemma 2.6]{Huang-2018}, it is proved that for a Cayley graph over an abelian group $G$ with Sylow 2-subgroup $P=\langle a_1 \rangle \times \dots \times \langle a_n \rangle$, where $o(a_i)=2^{m_i}>1$ for $1\le i\le n$, if $H$ is a subgroup perfect code, then either $H\cap P$ is trivial or $H\cap P$ projects onto at least one of $\langle a_1\rangle,\ldots,\langle a_n\rangle$. The following example shows that this result does not hold for bi-Cayley graphs in general.
\begin{example}\label{ProjectEx}
  Let $G=\langle a\rangle \times \langle b\rangle$, where both $a$ and $b$ have order 4. Then the Sylow 2-subgroup of $G$ is $P=G$. Let $H=\langle a^{2}\rangle$, and define
$$R=L=\{b^{2}\}, \qquad
S=\{a,ab,ab^{2},ab^{3},b,b^{3}\}.$$
It is straightforward to verify that both $R\cup S^{-1}\cup{1_G}$ and $L\cup S\cup{1_G}$ are left transversals of $H$ in $G$. Hence, by Corollary~\ref{LeftTransversal}, $H$ is a subgroup perfect code of $\BC(G;R,L,S)$. On the other hand,
$
H\cap P=\langle a^{2}\rangle
$
which projects onto neither $\langle a\rangle$ nor $\langle b\rangle$. 
\end{example}
In contrast Example \ref{ProjectEx}, if the set $S$ in a bi-Cayley graph $\BC(G;R,L,S)$ is empty, then that graph would be the disjoint union of two Cayley graphs. This leads to the following lemma. 
\begin{lemma}\label{OneCyclic}
  Let $G$ be an abelian group, and let $H,K\leq G$. Suppose that
$P=\langle a_1\rangle\times\cdots\times\langle a_n\rangle$
is a Sylow $2$-subgroup of $G$, where $o(a_i)=2^{m_i}$ for $1\leq i\leq n$. If $H_0\cup K_1$ is a subgroup perfect code of $\BC(G;R,L,\emptyset)$, then each of $H\cap P$ and $K\cap P$ is either trivial or projects onto at least one of the $\langle a_i\rangle$.
\end{lemma}
\begin{proof}
Let $\Gamma=\BC(G;R,L,\emptyset)$, and suppose that $H_0\cup K_1$ is a subgroup perfect code in $\Gamma$. Since $S=\emptyset$, we can view $\Gamma$ as the disjoint union of two Cayley graphs,
$\Gamma_0=\Cay(G_0,R)$ and $\Gamma_1=\Cay(G_1,L)$.
Since $H_0$ is a subgroup perfect code in $\Gamma_0$, \cite[Lemma 2.6]{Huang-2018} implies that $H\cap P$ is either trivial or projects onto at least one direct factor $\langle a_i\rangle$. Similarly, since $K_1$ is a subgroup perfect code in $\Gamma_1$, \cite[Lemma 2.6]{Huang-2018} implies that $K\cap P$ is either trivial or projects onto at least one direct factor $\langle a_i\rangle$.
\end{proof}
We note that the condition in Lemma~\ref{OneCyclic} is not sufficient, as illustrated by the following example.
\begin{example}\label{ExOneCyclic}
 Let $G=\langle a \rangle \times \langle b \rangle$, where both $a$ and $b$ have order 4. Set 
 $H=\langle a, b^2 \rangle$.
Then, $\pi _1(H)=\langle a \rangle$. Take $g_0=b_0\in G_0$. Then, $(g^2)_0=(b^2)_0\in H_0$. On the other hand, every element $h\in H$ has the form $h=a^ib^{2j}$, where $i\in \{0, 1,2,3\}$ and $j\in \{0,1\}$. Hence,  
$gh=a^ib^{2j+1}$,
 and therefore 
 $(gh)^2=a^{2i}b^2\neq 1_G$. Therefore, by Theorem~\ref{Perfect-Property}, $H$ cannot be a perfect code in any bi-Cayley graph.
\end{example}
We show that the condition in Lemma \ref{OneCyclic} is sufficient whenever $H\cap P$ is cyclic.
\begin{theorem}\label{ThCyclic}
  Let $G$ be an abelian group and $H, K\leq G$. Suppose that $P=\langle a_1 \rangle \times \dots \times \langle a_n \rangle $ is a Sylow 2-subgroup of $G$ with $o (a_i)=2^{m_i}$ for $1\leq i\leq n$.
   Let $H\cap P$ and $K\cap P$ be cyclic. Then, $H_0\cup K_1$ is a subgroup perfect code of a bi-Cayley graph $\BC(G; R,L, \emptyset)$ if and only if each $H\cap P$ and $K\cap P$ is either trivial or  projects onto at least one of the $\langle a_i \rangle$. 
\end{theorem} 
\begin{proof}
The necessity holds by Lemma \ref{OneCyclic}.

To prove the sufficiency, let $\Gamma=\BC(G;R,L,\emptyset)$. Since $S=\emptyset$, $\Gamma$ is a disjoint union of two Cayley graphs
$\Gamma_0=\Cay(G_0,R)$ and $\Gamma_1=\Cay(G_1,L)$.
By \cite[Theorem 2.7]{Huang-2018}, $H_0$ is a  perfect code of $\Gamma_0$ and $K_1$ is a  perfect code of $\Gamma_1$. Therefore, $H_0\cup K_1$ is a  perfect code of $\Gamma$.
\end{proof}
\begin{corollary}
  Suppose that $G$ is a cyclic group and $H, K\leq G$. Then, $H_0 \cup K_1$ is a subgroup perfect code of some bi-Cayley graph $\BC(G; R,L, \emptyset)$ if and only if either $|H|$ or $|G/H|$ is odd, and either $|K|$ or $|G/K|$ is odd.
\end{corollary}
\begin{proof}
  By Theorem \ref{ThCyclic}, $H$ is a subgroup perfect code of $G$ if and only if $H\cap P$ is either trivial or equal to $P$, where $P$ is the Sylow 2-subgroup of $G$.  Moreover, $H\cap P$ is trivial if and only if $|H|$ is odd, while $H\cap P=P$ if and only if $|G/H|$ is odd. The same statements hold for $K$.
\end{proof}
\subsection{Dihedral Groups}\label{Dihedral}
 Let
$D_{2n}=\langle a,b\mid a^n=b^2=(ab)^2=e\rangle$
be the dihedral group of order $2n$, where $n\geq 3$.
The subgroups of $D_{2n}$ are of the form $\langle a^t\rangle$, where $t$ divides $n$, or
$\langle a^t,a^s b\rangle$,
where $t$ is a divisor of $n$ and $0\leq s\leq t-1$. In contrast to the case of Cayley graphs, where not every subgroup is a perfect code \cite[Theorem 2.11]{Huang-2018}, every subgroup of $D_{2n}$ can occur as a subgroup perfect code in a bi-Cayley graph. In particular, if $H\leq D_{2n}$, then $H_0\cup H_1$ is a perfect code in some regular graph; see \cite[Corollary 5.3]{Wang-2026}. Here, we are interested in the more general case of subgroup perfect codes of the form $H_0\cup K_1$, where $H,K\leq D_{2n}$ and $H\neq K$. 
\begin{pro}\label{Di-Pro1}
Let $H=\{1,b\}$ and $K=\{e\}$ be two subgroups of $D_{2n}$. Then, $H_0\cup K_1$ is a subgroup perfect code of $D_{2n}$.
\end{pro}
\begin{proof}
  Set $S=\{a, a^{n-1}b\}$, $R=\{a^ib~|~1\leq i \leq n-2\}$, and $L=D_{2n}\setminus\{1, a, a^{n-1}, ab, a^{n-1}b\}$.  It is straightforward to verify that \eqref{Construction} holds.
\end{proof}
Proposition \ref{Di-Pro1} shows that Proposition \ref{AbeliSEmpty} does not extend to nonabelian groups.
\begin{pro}
Let $H=\{1,b\}$ and $K=\{1, a^{n/2}b\}$ be two subgroups of $D_{2n}$, where $n$ is an even number. Then, $H_0\cup K_1$ is a subgroup perfect code of $D_{2n}$.
\end{pro}
\begin{proof}
   We let $S=\langle a \rangle \setminus\{1, a^{n/2}\}$, $R=\{a^{n/2}b\}$, and $L=\{a^{n/2}\}$.  It is easy to verify that \eqref{Construction} holds.
\end{proof}

\section{Concluding Remarks}\label{Conclude}
In this paper, we extended the equivalent condition known for subgroup perfect codes in Cayley graphs to the setting of bi-Cayley graphs. Based on this condition, we studied subgroup perfect codes of the form $H_0\cup K_1$ over a finite group $G$, considering both the general case and the case where $G$ is abelian. We also obtained two families of perfect codes in dihedral groups, which provide examples in the nonabelian setting.

There are several interesting directions for future research. The equivalent condition obtained in this paper does not take into account the structural properties of the corresponding bi-Cayley graph, such as connectivity. It would therefore be interesting to investigate how such graph-theoretic properties can be incorporated into the characterization of subgroup perfect codes. More generally, subgroup perfect codes in nonabelian groups remain a relatively unexplored topic and deserve further study. Finally, a complete classification of subgroup perfect codes of the form $H_0\cup K_1$, where $H\neq K$, in connected bi-Cayley graphs over $D_{2n}$ remains an open problem.

  \section*{Acknowledgement}
  The first author was supported by  Alzahra University, Tehran, Iran.

\end{document}